\documentclass[12pt,twoside,leqno]{article}
\usepackage[T1]{fontenc}
\usepackage{amsfonts}
\usepackage{amsmath,amsthm}
\usepackage{amssymb,latexsym}
\usepackage{mathrsfs} 
\usepackage{graphics,graphicx}
\theoremstyle{plain}
\newtheorem{theorem}{Theorem}[section]
\newtheorem{lemma}[theorem]{Lemma}
\newtheorem{proposition}[theorem]{Proposition}
\newtheorem{corollary}[theorem]{Corollary}

\theoremstyle{definition}
\newtheorem{example}[theorem]{Example}
\newtheorem{definition}[theorem]{Definition}
\theoremstyle{remark}
\newtheorem{remark}[theorem]{Remark}
\newtheorem{question}[theorem]{Question}

\DeclareMathOperator{\Iso}{{\rm Iso}}

\DeclareMathOperator{\inter}{{\rm int}}
\DeclareMathOperator{\cl}{{\rm cl}}

\DeclareMathOperator{\lsim}{\lesssim}
\DeclareMathOperator{\super}{{\rm Super}}
\begin{document}
\title{Quasiorders for a characterization of iso-dense spaces}
\author{Thomas Richmond and Eliza Wajch\\
Department of Mathematics\\
Western Kentucky University,\\
Bowling Green, KY 42101, USA\\
email: tom.richmond@wku.edu\\
Institute of Mathematics,\\
University of Siedlce\\
3 Maja 54, 08-110, Siedlce, Poland\\
email: eliza.wajch@gmail.com}
\maketitle
\begin{abstract}

A (generalized) topological space is called an iso-dense space if the set of all its isolated points is dense in the space. The main aim of the article is to show in $\mathbf{ZF}$ a new characterization of iso-dense spaces in terms of special quasiorders. 
For a non-empty family $\mathcal{A}$ of subsets of a set $X$, a quasiorder $\lsim_{\mathcal{A}}$ on $X$ determined by $\mathcal{A}$ is defined. Necessary and sufficient conditions for $\mathcal{A}$ are given to have the property that the topology consisting of all $\lsim_{\mathcal{A}}$-increasing sets coincides with the generalized topology on $X$ consisting of the empty set and all supersets of non-empty members of $\mathcal{A}$. The results obtained, applied to the quasiorder $\lsim_{\mathcal{D}}$ determined by the family $\mathcal{D}$ of all dense sets of a given (generalized) topological space, lead to a new characterization of non-trivial iso-dense spaces. Independence results concerning resolvable spaces are also obtained.
\medskip


\end{abstract}

\maketitle

\section{Introduction}
\label{Intro}

The set-theoretic framework for this paper is the Zermelo-Fraenkel system of axioms $\mathbf{ZF}$. The Axiom of Choice ($\mathbf{AC}$) is not an axiom of $\mathbf{ZF}$. The system $\mathbf{ZF+AC}$ is denoted by $\mathbf{ZFC}$.

 We recall that a \emph{quasiorder} (or, equivalently, a \emph{preorder}) on a set $X$ is a reflexive and transitive binary relation on $X$. For a given a quasiorder $\lsim$ on $X$, the family 
 
\[\tau[\lsim]=\{U\subseteq X: (\forall x\in U)(\forall y\in X)(x\lsim y\rightarrow y\in U) \}\]
is a topology on $X$ called the specialization topology from $\lsim$ (see \cite[p. 195]{Rich}).

Every family of subsets of a given set $X$ determines a quasiorder on $X$ in the sense of the following definition: 

\begin{definition}
\label{s1d1}
Let $\mathcal{A}$ be a family of subsets of a set $X$. The binary relation $\lsim_{\mathcal{A}}$ on $X$ defined by the following rule:
\[ (\forall x,y\in X) (x\lsim_{\mathcal{A}} y\leftrightarrow (\forall A\in \mathcal{A})(x\in A\rightarrow y\in A))\]
is called the \emph{quasiorder determined by} $\mathcal{A}$.
\end{definition}

If $\tau$ is a topology on $X$, the quasiorder $\lsim_{\tau}$ has been considered by various authors. It is usually called the \emph{specialization}(or \emph{canonical}) \emph{preorder} of the topological space $\langle X,\tau\rangle$. Basic properties of specialization preorders can be found in \cite{Rich}.

Every family of subsets of a set $X$ is a base of a generalized topology on $X$ in the sense of the following definitions:

\begin{definition}
\label{s1d2}
\begin{enumerate}
\item[(i)] A \emph{generalized topology} on a set $X$ is a family $\mu$ of subsets of $X$ such that $\emptyset\in \mu$ and, for every subfamily $\mathcal{U}$ of $\mu$, $\bigcup\mathcal{U}\in\mu$ (cf. \cite{Cs}).
\item[(ii)] A \emph{strong generalized topology} on a set $X$ is a generalized topology $\mu$ on $X$ such that $X\in\mu$ (cf. \cite{Cs2, Cs3}).
\item[(iii)] A \emph{base for a generalized topology} $\mu$ on a set $X$ is a subfamily $\mathcal{B}$ of $\mu$ satisfying the following condition:
\[ (\forall U\in\mu)(\forall x\in U)(\exists B\in\mathcal{B}) x\in B\subseteq U\]
\noindent(cf. \cite{Cs1, Cs3}).
\item[(iv)] A \emph{generalized topological space} is an ordered pair $\langle X, \mu\rangle$ where $X$ is a set and $\mu$ is a generalized topology on $X$ (cf.\cite{Cs}). If $\mu$ is a strong generalized topology on $X$, the generalized topological space $\langle X, \mu\rangle$ is said to be \emph{strong}.
\end{enumerate}
\end{definition}

Strong generalized topologies are called \emph{supra topologies} in \cite{MAMK}. A strong generalized topology $\mu$ on $X$ is a topology on $X$ if and only if $\mu$ is closed under finite intersections. Generalized topologies have been widely studied by many mathematicians (see, for instance, \cite{Cs, Cs0, Cs1, Cs2, Cs3, HTW, MAMK, Min, TCh}).

\begin{definition}
\label{s1d3}
Let $\mathcal{A}$ be a family of subsets of a set $X$.
\begin{enumerate}
\item[(i)] The \emph{generalized topology determined by} $\mathcal{A}$ is the family $\mu[\mathcal{A}]$ defined as follows:
\[\mu[\mathcal{A]}:=\{U\subseteq X: (\forall x\in U)(\exists A\in\mathcal{A}) x\in A\subseteq U\}.\]
\item[(ii)] The \emph{extended generalized topology determined by} $\mathcal{A}$ is the family $\tilde{\mu}[\mathcal{A}]$ defined as follows:
\[\tilde{\mu}[\mathcal{A}]:=\mu[\mathcal{A}]\cup\{V\subseteq X: (\exists A\in\mathcal{A}\setminus\{\emptyset\}) A\subseteq V\}.\]
\end{enumerate}
\end{definition}

Let us notice that if $\mathcal{A}$ is a family of subsets of a set $X$, then $\mathcal{A}$ is a base for the generalized topology $\mu[\mathcal{A}]$ on $X$. Furthermore, if $\emptyset\neq\mathcal{A}\neq\{\emptyset\}$, the extended generalized topology $\tilde{\mu}[\mathcal{A}]$ is strong. 

In Section \ref{s3}, given a non-empty family $\mathcal{A}$ of subsets of a set $X$ such that $\mathcal{A}\neq\{\emptyset\}$, we show necessary and sufficient conditions for $\mathcal{A}$ to have the property that $\tau[\lsim_{\mathcal{A}}]=\mu[\mathcal{A}]$ (see Theorem \ref{s3t7}). We also give a number of conditions under which $\tau[\lsim_{\tilde{\mu}[\mathcal{A}]}]=\tilde{\mu}[\mathcal{A}]$ (see Theorem \ref{s3t10}).  Some of our results are relevant to the well-known characterization of Alexandroff spaces via quasiorders. Let us recall the definition of an Alexandroff space, which has its roots in \cite{Alex}.

\begin{definition}
\label{s1d4}
(Cf. \cite[p. 17]{A} and \cite[Definition 8.1.1]{Rich}.) A topological space $\mathbf{X}$ is called an \emph{Alexandroff space} if, for every non-empty family $\mathcal{U}$ of open sets in $\mathbf{X}$, the set $\bigcap\mathcal{U}$ is open in $\mathbf{X}$. The topology of an Alexandroff space is called an \emph{Alexandroff topology}.
\end{definition}

In the light of \cite[Theorem 8.3.3]{Rich}, we have the following characterization of Alexandroff topologies:

\begin{theorem}
\label{s1t5}
(Cf. \cite[Theorem 8.3.3]{Rich}.) For every topology $\tau$ on a set $X$, the following conditions are equivalent:
\begin{enumerate}
\item[(i)] $\tau$ is an Alexandroff topology;
\item[(ii)] there exists a quasiorder $\lsim$ on $\mathbf{X}$ such that $\tau=\tau[\lsim]$;
\item[(iii)] $\tau=\tau[\lsim_{\tau}]$.
\end{enumerate}
\end{theorem}

Before we pass to the body of this paper, in Section \ref{s2}, we establish our notation and terminology. At the end of Section \ref{s2}, we give an illuminating Example \ref{s2e12} which, together with Theorem \ref{s1t5}, has motivated us to this work. In Section \ref{s3}, among other helpful things, we slightly modify Theorem \ref{s1t5} by observing that, for every generalized topology $\mu$ on a non-empty set $X$, it holds that $\mu=\tau[\lsim_{\mu}]$ if and only if $\mu$ is an Alexandroff topology (see Theorem \ref{s3t7}).

In Section \ref{s4}, we apply the results obtained in Section \ref{s3} to a new characterization of iso-dense spaces. The term ``iso-dense'' was introduced in \cite{KTW1} in the sense of the following definition:

\begin{definition}
\label{s1d6}
Let $\mathbf{X}$ be a topological space and let $\Iso(\mathbf{X})$ be the set of all isolated points of $\mathbf{X}$.
\begin{enumerate}
\item[(i)] (Cf. \cite{KTW1}.) The space $\mathbf{X}$ is called an \emph{iso-dense space} if the set $\Iso(\mathbf{X})$ is dense in $\mathbf{X}$. 
\item[(ii)] (Cf. \cite[Chapter 1.3]{En}.) If $\Iso(\mathbf{X})=\emptyset$, then $\mathbf{X}$ is called a \emph{dense-in-itself} (or \emph{crowded}) space.
\end{enumerate}
\end{definition}

Let us remark that, by \cite[Proposition 5]{KTW1}, every scattered space is iso-dense. Clearly, every discrete space is iso-dense. Every compactification of an infinite discrete space is iso-dense. Indiscrete spaces consisting of at least two points are trivial dense-in-itself spaces. 

For the aims of Section \ref{s4}, we extend, in Section \ref{s2}, the concept of an iso-dense space to generalized topological spaces having dense sets of isolated points (see items (4) and (7) of Definition \ref{s2d1}). In Section \ref{s4}, given a generalized topological space $\mathbf{X}=\langle X, \mu\rangle$, we consider the family $\mathcal{D}(\mathbf{X})$ of all dense subsets of $\mathbf{X}$ (see item (5) of Definition \ref{s2d1}). Then the generalized topology determined by $\mathcal{D}(\mathbf{X})$ is the family $\mu[\mathcal{D}(\mathbf{X})]=\mathcal{D}(\mathbf{X})\cup\{\emptyset\}$. The generalized topology $\mu[\mathcal{D}(\mathbf{X})]$ is not necessarily a topology on $X$. The following two questions arise:

\begin{question}
\label{s1q7}
Under which conditions on a (generalized) topological space $\mathbf{X}$ is $\mu[\mathcal{D}(\mathbf{X})]$ a topology?
\end{question}

\begin{question}
\label{s1q8}
Under which conditions on a (generalized) topological space $\mathbf{X}$ is $\mu[\mathcal{D}(\mathbf{X})]$ an Alexandroff topology?
\end{question}

Obviously, an answer to Question \ref{s1q8} is a partial answer to Question \ref{s1q7}. Our main goal is to answer Question \ref{s1q8} and apply it to a new characterization of non-trivial iso-dense spaces in Section \ref{s4}. To this aim, we pay attention to the quasiorder $\lsim_{\mathcal{D}(\mathbf{X})}$ and the topology $\tau[\lsim_{\mathcal{D}(\mathbf{X})}]$.  Theorems \ref{s4t8} and \ref{s4t9} are the main results of Section \ref{s4}. Theorem \ref{s4t8} characterizes, in terms of $\lsim_{\mathcal{D}(\mathbf{X})}$, non-trivial iso-dense generalized topological spaces. Namely, we show in Theorem \ref{s4t8} that, for every non-indiscrete generalized topological space $\mathbf{X}$, the following are all equivalent: (i) $\mathbf{X}$ is iso-dense, (ii) $\tau[\lsim_{\mathcal{D}(\mathbf{X})}]=\mu[\mathcal{D}(\mathbf{X})]$, (iii) $\mu[\mathcal{D}(\mathbf{X})]$ is an Alexandroff topology. This answers Question \ref{s1q8}. We also notice that, for every non-indiscrete but dense-in-itself generalized topological space $\mathbf{X}$, the topology $\tau[\lsim_{\mathcal{D}(\mathbf{X})}]$ is discrete (see Theorem \ref{s4t7}). We modify these results to get a characterization of all non-indiscrete generalized topological spaces $\mathbf{X}$ having the property that the set of all not nowhere dense singletons of $\mathbf{X}$ is both dense and open in $\mathbf{X}$ (see Theorem \ref{s4t9}). This leads to other necessary and sufficient conditions for a generalized topological $T_1$-space $\langle X, \mu\rangle$ with $\mu\neq\{\emptyset\}$ to be iso-dense (see Corollary \ref{s4c12}).

Although the main results of Section \ref{s4} answer Question \ref{s1q8}, they are not satisfactory answers to Question \ref{s1q7}. To give a little  more light into Question \ref{s1q7} in Section \ref{s5}, we apply the following concept of a resolvable space introduced by Hewitt in \cite{Hew}:

\begin{definition}
\label{s1d9}
(Cf. \cite{Hew}.) A topological space $\mathbf{X}$ is called a \emph{resolvable space} if there exists a pair of disjoint dense sets in $\mathbf{X}$. Topological spaces which not resolvable are called \emph{irresolvable spaces}.
\end{definition}

Both classes of resolvable and irresolvable topological spaces have been widely studied in $\mathbf{ZFC}$ for years (see, e.g., \cite{Hew}, \cite{Pav}, \cite{RST} and \cite{DLRT}), but, to the best of our knowledge, they have not been investigated in $\mathbf{ZF}$ so far. We generalize Definition \ref{s1d9} to the concepts of a resolvable and an irresolvable generalized topological space (see items (13) and (14) of Definition \ref{s2d1}). Proposition \ref{s5p4} shows that if a non-indiscrete generalized topological space $\mathbf{X}=\langle X, \mu\rangle$ is resolvable, then $\mu[\mathcal{D}(\mathbf{X})]$ is not a topology. Therefore, to give a satisfactory answer to Question \ref{s1q7}, it is necessary to investigate the class of irresolvable generalized topological spaces in $\mathbf{ZF}$. To point out that the situation of this class in $\mathbf{ZF}$ significantly differs from that in $\mathbf{ZFC}$, we observe that even the following statement ``for every infinite set $X$ and the cofinite topology $\tau_{cof}(X)$ on $X$ (see Definition \ref{s2d14}), the space $\langle X, \tau_{cof}(X)\rangle$ is resolvable'', known to be true in $\mathbf{ZFC}$ (see \cite[p. 3]{RST}), is unprovable in $\mathbf{ZF}$ (see Theorem \ref{s5t11} and Corollary \ref{s5c12}). We leave deeper research on irresolvability in $\mathbf{ZF}$ and a more satisfactory answer to Question \ref{s1q7} for another article. For the convenience of readers, we include a short list of open problems in Section~\ref{s6}. 

\section{Preliminaries}
\label{s2}

We use standard set-theoretic notation. For a set $X$, $[X]^{<\omega}$ is the set of all finite subsets of $X$. The power set of $X$ is denoted by $\mathcal{P}(X)$. The symbols $\mathbb{R}$, $\mathbb{Z}$, $\mathbb{N}$ denote, respectively, the set of all real numbers, the set of all integers and the set of all positive integers. The symbol $\leq$ stands for the standard linear order on $\mathbb{R}$.

All topological notions used in this article, if not introduced here, are standard and can be found in \cite{En} or \cite{Rich}. 

Throughout this paper, if not stated otherwise, we denote (generalized) topological spaces with boldface letters and their underlying sets with lightface letters. 

\begin{definition}
\label{s2d1}  
Let  $\mathbf{X}=\langle X, \mu\rangle$ be a given (generalized) topological space.
\begin{enumerate}
\item Members of $\mu$ are said to be $\mu$-\emph{open sets} (or, simply, \emph{open sets} in $\mathbf{X}$). A subset $C$ of $X$ such that $X\setminus C$ is $\mu$-open is said to be $\mu$-\emph{closed} (or \emph{closed} in $\mathbf{X}$).
\item For a set $E\subseteq X$, $\cl_{\mathbf{X}}(E)$ denotes the closure of $E$ in $\mathbf{X}$ (that is, the intersection of all $\mu$-closed sets containing $E$), and $\inter_{\mathbf{X}}(E)$ denotes the interior of $E$ in $\mathbf{X}$ (that is, the union of all $\mu$-open sets contained in $E$).
\item A subset $E$ of $X$ is called \emph{nowhere dense} in $\mathbf{X}$ (or, equivalently, $\mu$-\emph{nowhere dense}) if $\inter_{\mathbf{X}}(\cl_{\mathbf{X}}(E))=\emptyset$ (cf. \cite{TCh}). The collection of all $\mu$-nowhere dense sets is denoted by $\mathcal{ND}(\mathbf{X})$. 
\item  We say that a point $x\in X$ is an \emph{isolated point} of $\mathbf{X}$ (or a $\mu$-\emph{isolated point}) if $\{x\}\in\mu$. The set of all $\mu$-isolated points is denoted by $\Iso(\mu)$ or $\Iso(\mathbf{X})$.
\item A set $D\subseteq X$ will be called $\mu$-\emph{dense} (or \emph{dense} in $\mathbf{X}$) if for every $U\in\mu\setminus\{\emptyset\}$, $U\cap D\neq\emptyset$.  The collection of all $\mu$-dense sets is denoted by $\mathcal{D}(\mathbf{X})$ or by $\mathcal{D}(\mu)$.
\item $\mathcal{DO}(\mathbf{X}):=\mu\cap\mathcal{D}(\mathbf{X})$.
\item The space $\mathbf{X}$ is called an \emph{iso-dense (generalized) topological space} if $\Iso(\mathbf{X})\in\mathcal{D}(\mathbf{X})$. 
\item If $\Iso(\mathbf{X})=\emptyset$, $\mathbf{X}$ is called a \emph{dense-in-itself (generalized) topological space}.
\item $\mathbf{X}$ is called a $T_0$-\emph{space} if, for every pair $x,y$ of distinct points of $X$, there exists $U\in \mu$ such that $U\cap\{x, y\}$ is a singleton (cf. \cite{Cs0, MAMK}).
\item $\mathbf{X}$ is called a $T_1$-\emph{space} if every singleton of $\mathbf{X}$ is $\mu$-closed (cf. \cite{Cs0, MAMK}).
\item $\mathbf{X}$ is called an \emph{indiscrete space} or \emph{trivial space} if $\mu\subseteq\{\emptyset, X\}$.
\item $\mathbf{X}$ is called a \emph{discrete space} if $\mu=\mathcal{P}(X)$.
\item $\mathbf{X}$ is called a \emph{resolvable space} if there exists a set $D\in\mathcal{D}(\mathbf{X})$ such that $X\setminus D\in\mathcal{D}(\mathbf{X})$. If $\mathbf{X}$ is a resolvable space, the generalized topology $\mu$ is said to be \emph{resolvable}.
\item $\mathbf{X}$ is called an \emph{irresolvable space} if it is not a resolvable space.  If $\mathbf{X}$ is an irresolvable space, the generalized topology $\mu$ is said to be \emph{irresolvable}.
\end{enumerate}
\end{definition} 

\begin{remark}
\label{s2r2}
Let $\mu$ be a given generalized topology on a set $X$ and let $\tau$ be the coarsest topology on $X$ containing $\mu$. Then every $\mu$-isolated point is $\tau$-isolated. A point $x\in X$ is $\tau$-isolated if and only if there exists $\mathcal{U}\in [\mu\cup\{X\}]^{<\omega}\setminus\{\emptyset\}$ such that $\{x\}=\bigcap\mathcal{U}$. Every $\tau$-dense subset of $X$ is $\mu$-dense. A subset $D$ of $X$ is $\tau$-dense if and only if, for every $\mathcal{U}\in [\mu\cup\{X\}]^{<\omega}\setminus\{\emptyset\}$, $\bigcap\mathcal{U}\neq\emptyset$ implies $D\cap\bigcap\mathcal{U}\neq\emptyset$.
\end{remark}

\begin{remark}
\label{s2r3}
For a generalized topological space $\mathbf{X}$, the family $\mathcal{D}(\mathbf{X})\cup\{\emptyset\}$ is a generalized topology but not necessarily a topology. Clearly, $\mathcal{D}(\mathbf{X})\cup\{\emptyset\}$ is a topology if and only if, for every pair $A, B$ of dense sets in $\mathbf{X}$, the intersection $A\cap B$ is either empty or dense in $\mathbf{X}$. 
\end{remark}

In view of Remark \ref{s2r3}, the following question is equivalent to Question \ref{s1q7}.

\begin{question}
\label{s2q4}
Under which conditions on a (generalized) topological space $\mathbf{X}$ is the intersection of any two dense sets in $\mathbf{X}$ either empty or dense in $\mathbf{X}$?
\end{question} 

If a generalized topological space $\mathbf{X}$ is not strong, $\mathcal{DO}(\mathbf{X})\cup\{\emptyset\}$ is not a topology on $X$. The following example shows that even for a strong generalized topological space $\mathbf{X}$, the generalized topology $\mathcal{DO}(\mathbf{X})\cup\{\emptyset\}$ need not be a topology.

\begin{example}
\label{s2e5}
For $X=\{1, 2, 3\}$, let $\mu=\{\emptyset, X, \{1,2\}, \{2, 3\}\}$ and $\mathbf{X}=\langle X, \mu\rangle$. Then $\mathcal{DO}(\mathbf{X})\cup\{\emptyset\}=\mu$ is not a topology. Since the sets $\{2\}$ and $\{1,3\}$ are dense in $\mathbf{X}$,  the generalized topological space $\mathbf{X}$ is resolvable.
\end{example}

The following proposition is immediate.

\begin{proposition}
\label{s2p6}
 Let $\mathbf{X}=\langle X, \mu\rangle$ be a strong generalized topological space such that, for every $U\in\mu$ and every $D\in\mathcal{DO}(\mathbf{X})$, $U\cap D\in\mu$. Then $\mathcal{DO}(\mathbf{X})\cup\{\emptyset\}$ is a topology on $X$. In particular, for every topological space $\mathbf{X}$, $\mathcal{DO}(\mathbf{X})\cup\{\emptyset\}$ is a topology on $X$.
\end{proposition}

In the notation of \cite[Definition 8.5.1]{Rich}, for a given set $X$ and its subset $S$, 
\[\super(S):=\{U\in\mathcal{P}(X): S\subseteq U\}\cup\{\emptyset\}.\]
The family $\super(S)$ is an Alexandroff topology on $X$ called the \emph{topology of surpersets} of $S$. If $S$ is a non-empty subset of $X$ and $\mathbf{X}=\langle X, \super(S)\rangle$, then $\mathbf{X}$ is irresolvable, $\mathcal{D}(\mathbf{X})\cup\{\emptyset\}=\super(S)$ and $\mathcal{D}(\mathbf{X})$ is closed under finite intersections. The following proposition holds.

\begin{proposition}
\label{s2p7}
For every (generalized) topological space $\mathbf{X}$,  $\mathcal{D}(\mathbf{X})\subseteq\super(\Iso(\mathbf{X}))$. Furthermore, a non-indiscrete (generalized) topological space $\mathbf{X}$ is iso-dense if and only if \[\super(\Iso(\mathbf{X}))=\mathcal{D}(\mathbf{X})\cup\{\emptyset\}.\]
\end{proposition}

Proposition  \ref{s2p7} also leads to Question \ref{s1q8} motivating this research. As we have already mentioned in Section \ref{Intro}, an answer to  Question \ref{s1q8} is given in Section \ref{s4}.

Since Alexandroff topologies are determined by quasiorders, we need to have a look at quasiorders. All one needs to know about quasiorders to understand the forthcoming sections can be found in \cite[Chapter 8]{Rich} and \cite{Nach}. We establish our notation and terminology below.

\begin{definition}
\label{s2d8}
Let $\lsim$ be a quasiorder on a set $X$.
\begin{enumerate}
\item For every $x\in X$, 
\[\uparrow[\lsim, x]:=\{y\in X: x\lsim y\}\text{ and } \downarrow[\lsim, x]:=\{y\in X: y\lsim x\}. \]
\item For all $x, y\in X$, $x\approx y$ means that $x\lsim y$ and $y\lsim x$.
\item A subset $P$ of $X$ is said to be $\lsim$-increasing (respectively, $\lsim$-decreasing) if, for all $x,y\in X$ such that $x\in P$ and $x\lsim y$ (respectively, $y\lsim x$), we have $y\in P$.
\item An element $a\in X$ is called $\lsim$-\emph{maximal} (respectively, $\lsim$-\emph{minimal}) if $\uparrow[\lsim, m]=\{a\}$ (respectively, $\downarrow[\lsim, a]=\{a\}$).
\item An element $a\in X$ is called \emph{weakly} $\lsim$-\emph{maximal} (respectively, \emph{weakly} $\lsim$-\emph{minimal}) if, for every $x\in X$ such that $a\lsim x$ (respectively, $x\lsim a$), we have $x\approx a$.
\item The \emph{dual quasiorder} from $\lsim$ is the quasiorder $\lsim^{d}$ defined as follows:
\[(\forall x,y\in X)(x\lsim^{d} y\leftrightarrow y\lsim x).\]
\item The \emph{dual specialization topology} from $\lsim$ is the topology $\tau[\lsim^{d}]$.
\end{enumerate}
\end{definition}
 
\begin{remark}
\label{s2r9}
Let  $\lsim$  be a given quasiorder on a set $X$. The specialization topology $\tau[\lsim]$ consists of all $\lsim$-increasing sets. The dual specialization topology $\tau[\lsim^d]$ consists of all $\lsim$-decreasing sets. Obviously, $\Iso(\tau[\lsim])$ is the set of all $\lsim$-maximal elements, and $\Iso(\tau[\lsim^d])$ is the set of all $\lsim$-minimal elements.
\end{remark}

\begin{remark}
\label{s2r10}
If $\lsim$ is a partial order, the notions of a weakly $\lsim$-maximal element and a $\lsim$-maximal element are equivalent (so are the notions of a weakly $\lsim$-minimal and a $\lsim$-minimal element). 
\end{remark}

\begin{remark} 
\label{s2r11}
Given a quasiorder $\lsim$ on a set $X$, the relation $\approx$ is an equivalence relation. One can define a partial order $\preceq$ on the set $X/_{\approx}$ of all equivalence classes of $\approx$ as follows: for all $a,b\in X/_{\approx}$, $a\preceq b$ if and only if, for each $x\in a$ and each $y\in b$, $x\lsim y$ (see \cite[Theorem 8.2.2]{Rich}).  Then an element $m\in X$ is weakly $\lsim$-maximal (respectively, weakly $\lsim$-minimal) if and only if the equivalence class $[m]_{\approx}$ of $\approx$ containing $m$ is $\preceq$-maximal (respectively, $\preceq$-minimal).
\end{remark}

The following example suggests some of the problems we consider.

\begin{example} 
\label{s2e12}
Consider the partial order on $\mathbb{R}$ defined as follows. For all $x,y\in\mathbb{R}$, $x \lsim y$ if and only if there exists $n \in \mathbb{Z}$ such that either $x, y \in [2n-1, 2n]$ and $y - x \geq 0$ or  $x, y \in [2n, 2n+1]$ and $x - y \geq 0$, as suggested in the figure below. 

\begin{figure}
\begin{center}
\setlength{\unitlength}{2pt}
 \begin{picture}(120,40)(0,-15)
 \put(2,5){$\cdots$}
 \put(20,0){\line(-1,1){5}}
 \multiput(20,0)(30,0){3}{\line(1,1){15}}
 \multiput(35,15)(30,0){3}{\line(1,-1){15}}
 \put(50,0){\circle*{2}}
  \put(65,15){\circle*{2}}
   \put(80,0){\circle*{2}}
    \put(110,0){\line(1,1){5}}
  \put(120,5){$\cdots$}
  \put(50, -7){\makebox[0pt]{$2n-1$}}
   \put(65, 19){\makebox[0pt]{$2n$}}
    \put(80, -7){\makebox[0pt]{$2n+1$}}
 \end{picture}
 \end{center}
\caption{An iso-dense space $\mathbf{X}$ for which $\mathcal{D}(\mathbf{X}) \cup \{\emptyset\}$ is a topology.} \label{fig1}
 \end{figure}
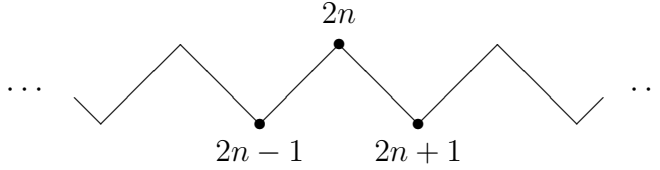

\noindent
Every $\lsim$-increasing set contains a $\lsim$-maximal element $2n$, so every non-empty $\tau[\lsim]$-open set intersects the set $M = \{2n : n \in \mathbb{Z}\}$ of all $\lsim$-maximal elements.  Thus, $M$ is dense in $\langle\mathbb{R}, \tau[\lsim]\rangle$. Since $M=\Iso(\tau[\lsim])$, the space $\langle\mathbb{R}, \tau[\lsim]\rangle$ is iso-dense. Therefore $\mathcal{D}(\tau[\lsim])\cup\{\emptyset\}=\super(\Iso(\tau[\lsim]))$. In this case, the collection of all $\tau[\lsim]$-dense sets is closed under finite intersections, and, for every element $x\in\mathbb{R}$, there exists a $\lsim$-maximal element $m$ such that $x\lsim m$.  
\end{example}

Let us briefly summarize what happens in a general situation similar to that in Example \ref{s2e12}.

\begin{proposition}
\label{s2p13}
Let $\langle X, \lsim\rangle$ be a quasiordered set, and let $M$ be the set of all $\lsim$-maximal elements. If, for every $x\in X$, there exists $m\in M$ such that $x\lsim m$, then $M=\Iso(\tau[\lsim])$, the space $\langle X, \tau[\lsim]\rangle$ is iso-dense, and $\mathcal{D}(\tau[\lsim])\cup\{\emptyset\}$ is the Alexandroff topology $\super(M)$.
\end{proposition}

In Section \ref{s5}, to discuss some difficulties with getting a satisfactory answer to Question \ref{s1q7} in $\mathbf{ZF}$, we apply cofinite topologies and amorphous sets. 

\begin{definition}
\label{s2d14}
For a set $X$, let
\[\tau_{cof}(X)=\{\emptyset\}\cup\{U\subseteq X: X\setminus U\in [X]^{<\omega}\}.\]
The topology $\tau_{cof}(X)$ is called the \emph{cofinite topology} on $X$.
\end{definition}

\begin{definition}
\label{s2d15}
An infinite set $X$ is called an \emph{amorphous set} if, for every infinite subset $Z$ of $X$, the set $X\setminus Z$ is finite. (See \cite[Note 57]{HR}.)
\end{definition}

\begin{remark}
\label{s2r16}
Form 64 of \cite{HR} is the statement: There are no amorphous sets. Since there are models of $\mathbf{ZF}$ in which Form 64 of \cite{HR} is false, it may happen in $\mathbf{ZF}$ that there are amorphous sets. For instance, Form 64 of \cite{HR} is false in Monro's Model III (model $\mathcal{M}37$ in \cite{HR}). Every statement proved to be equivalent to Form 64 of \cite{HR} in $\mathbf{ZF}$ is independent of $\mathbf{ZF}$.
 \end{remark}
 
In Theorem \ref{s5t11}, we show three new statements equivalent to \cite[Form 64]{HR} in $\mathbf{ZF}$. We conclude with the independence results in Corollary \ref{s5c12}. Apart from the independence results, all other results of this work are obtained in $\mathbf{ZF}$.

\section{The quasiorders and generalized topologies determined by families of sets}
\label{s3}

Throughout this section, we assume that $\mathcal{A}$ is a non-empty family of subsets of a set $X$ such that $\mathcal{A}\neq\{\emptyset\}$. 
The main aim of this section is to compare the generalized topologies $\mu[\mathcal{A}]$, $\tilde{\mu}[\mathcal{A}]$, $\tau[\lsim_{\mathcal{A}}]$ and $\tau[\lsim_{\tilde{\mu}[\mathcal{A}]}]$ (see Definitions \ref{s1d1} and \ref{s1d3}). The following illuminating example shows that these topologies can be pairwise distinct.

\begin{example}
\label{s3e1}
Let $X=\{1,2,3,4\}$ and $\mathcal{A}=\{\{1,2\}, \{2,3,4\}\}$. In this case, we have the following:
\begin{enumerate}
 \item[(i)] $\mu[\mathcal{A}]=\{\emptyset, X\}\cup\mathcal{A}$ and $\tilde{\mu}[\mathcal{A}]=\mu[\mathcal{A}]\cup\{ \{1,2,3\}, \{1,2,4\}\}$; 

\item[(ii)] $\uparrow[\lsim_{\mathcal{A}}, 1]=\{1,2\}$, $\uparrow[\lsim_{\mathcal{A}}, 2]=\{2\}$, $\uparrow[\lsim_{\mathcal{A}}, 3]=\uparrow[\lsim_{\mathcal{A}}, 4]=\{2,3,4\}$;

\item[(iii)] $\tau[\lsim_{\mathcal{A}}]=\mu[\mathcal{A}]\cup\{ \{2\}\}$;

\item[(iv)] $\uparrow[\lsim_{\tilde{\mu}[\mathcal{A}]},1]=\{1,2\}$, $\uparrow[\lsim_{\tilde{\mu}[\mathcal{A}]},2] =\{2\}$, $\uparrow[\lsim_{\tilde{\mu}[\mathcal{A}]},3]=\{2,3\}$ and $\uparrow~[\lsim_{\tilde{\mu}[\mathcal{A}]},4]=\{2, 4\}$;
\item[(v)]  $\tau[\lsim_{\tilde{\mu}[\mathcal{A}]}]=\tilde{\mu}[\mathcal{A}]\cup\{\{2\}, \{2,3\}, \{2,4\}\}$;
\item[(vi)]  $\mu[\mathcal{A}]$, $\tilde{\mu}[\mathcal{A}]$, $\tau[\lsim_{\mathcal{A}}]$ and $\tau[\lsim_{\tilde{\mu}[\mathcal{A}]}]$ are all irresolvable.
\end{enumerate}
\end{example}

\begin{definition}
\label{s3d2}
\begin{enumerate}
\item[(i)] $I(\mathcal{A}):=\bigcap(\mathcal{A}\setminus\{\emptyset\}).$
\item[(ii)] For every $x\in X$, $\mathcal{B}_{\mathcal{A}}(x):=\{A\in\mathcal{A}: x\in A\}$.
\item[(iii)] If $x\in X$ is such that $\mathcal{B}_{\mathcal{A}}(x)\neq\emptyset$, then $I_{\mathcal{A}}(x):=\bigcap\mathcal{B}_{\mathcal{A}}(x)$.
\end{enumerate}
\end{definition}

We state several simple facts in the following proposition.

\begin{proposition}
\label{s3p3}
\begin{enumerate}
\item $\tilde{\mu}[\mathcal{A}]=\{\emptyset\}\cup\{U\subseteq X: (\exists A\in\mathcal{A}\setminus\{\emptyset\})A\subseteq U\}$.
\item $I(\mathcal{A})=I(\mu[\mathcal{A}])=I(\tilde{\mu}[\mathcal{A}])$.
\item $I(\mathcal{A})\in\mu[\mathcal{A}]$ if and only if $I(\mathcal{A})\in\tilde{\mu}[\mathcal{A}]$.
\item For every $x\in X$, we have $\mathcal{B}_{\mathcal{A}}(x)\neq\emptyset$ if and only if $\mathcal{B}_{\mu[\mathcal{A}]}(x)\neq\emptyset$.
\item For every $x\in X$, if $\mathcal{B}_{\mathcal{A}}(x)\neq\emptyset$, then $\mathcal{B}_{\tilde{\mu}[\mathcal{A}]}(x)\neq\emptyset$ and \[I_{\tilde{\mu}[\mathcal{A}]}(x)\subseteq I_{\mathcal{A}}(x)=I_{\mu[\mathcal{A}]}(x).\]
\end{enumerate}
\end{proposition}

\begin{proposition}
\label{s3p4}
\begin{enumerate}
\item For all $x,y\in X$, we have:
 \[x\lsim_{\mathcal{A}}y\leftrightarrow \mathcal{B}_{\mathcal{A}}(x)\subseteq\mathcal{B}_{\mathcal{A}}(y).\]
 \item $\lsim_{\tilde{\mu}[\mathcal{A}]}\subseteq \lsim_{\mathcal{A}}=\lsim_{\mu[\mathcal{A}]}.$
 \item For every $x\in X$ such that $\mathcal{B}_{\mathcal{A}}(x)\neq\emptyset$, we have $I_{\mathcal{A}}(x)=\uparrow[\lsim_{\mathcal{A}}, x].$ 
 \item For every $z\in I(\mathcal{A})$, $I(\mathcal{A})=\uparrow[\lsim_{\mathcal{A}}, z]$. In consequence, $I(\mathcal{A})\in\tau[\lsim_{\mathcal{A}}]$. 
 \item $\mu[\mathcal{A}]\subseteq\tau[\lsim_{\mathcal{A}}]$. 
\end{enumerate}
\end{proposition}

\begin{proof}
We omit the elementary, simple proofs of (1)--(4). Since $\tau[\lsim_{\mathcal{A}}]$ is a topology, for the proof of (5), it suffices to show that $\mathcal{A}\subseteq\tau[\lsim_{\mathcal{A}}]$. Consider any non-empty set $A\in\mathcal{A}$ and any $z\in A$. It is easily seen that $\uparrow[\lsim_{\mathcal{A}}, z]\subseteq A$, so $A\in \tau[\lsim_{\mathcal{A}}]$.
\end{proof}

The following lemma is widely known and can be deduced immediately from \cite[Theorem 8.3.6]{Rich} by considering the identity function on $X$.

\begin{lemma}
\label{s3l5}
Let $\lsim$  and $\lsim^{\ast}$ be quasiorders on a set $X$. Then  $\lsim^{\ast}\subseteq \lsim$ if and only if $\tau[\lsim]\subseteq\tau[\lsim^{\ast}]$.
\end{lemma}

\begin{theorem}
\label{s3t6} 
\begin{enumerate}
\item $\tau[\lsim_{\mathcal{A}}]=\tau[\lsim_{\mu[\mathcal{A}]}]\subseteq\tau[\lsim_{\tilde{\mu}[\mathcal{A}]}]$.
\item The equality $\tau[\lsim_{\mathcal{A}}]=\tau[\lsim_{\tilde{\mu}[\mathcal{A}]}]$ holds if and only if $\lsim_{\mathcal{A}}\subseteq\lsim_{\tilde{\mu}[\mathcal{A}]}$.
\end{enumerate}
\end{theorem}
\begin{proof}
This follows directly from Proposition \ref{s3p4}(2) and Lemma \ref{s3l5}.
\end{proof}

Now, we are in a position to state the main results of this section.

\begin{theorem} 
\label{s3t7}
The following conditions are all equivalent:
\begin{enumerate}
\item[(i)] $\bigcup\mathcal{A}=X$ and, for every $x\in X$, $I_{\mathcal{A}}(x)\in\mu[\mathcal{A}]$;
\item[(ii)] $\tau[\lsim_{\mathcal{A}}]=\mu[\mathcal{A}]$;
\item[(iii)] $\mu[\mathcal{A}]$ is an Alexandroff topology on $X$. 
\end{enumerate}
\end{theorem}
\begin{proof}
Of course, if $\mu[\mathcal{A}]$ is a topology on $X$, then $\bigcup\mathcal{A}=X$. Thus, it follows from items (3) and (5) of Proposition \ref{s3p4} that conditions (i) and (ii) are equivalent. Since $\tau[\lsim_{\mathcal{A}}]$ is an Alexandroff topology, (ii) implies (iii).
It is obvious that (iii) implies (i).
\end{proof}

\begin{corollary}
\label{s3c8}
If the family $\mathcal{A}$ is finite, then $\tau[\lsim_{\mathcal{A}}]=\mu[\mathcal{A}]$ if and only if $\mu[\mathcal{A}]$ is a topology on $X$.
\end{corollary}

\begin{theorem}
\label{s3t9}
If $I(\mathcal{A})=\emptyset$, then the topology $\tau[\lsim_{\tilde{\mu}[\mathcal{A}]}]$ is discrete.
\end{theorem}

\begin{proof}
Let $x\in X$ and $V=~\uparrow[\lsim_{\tilde{\mu}[\mathcal{A}]}, x]$. Clearly, $V\in\tau[\lsim_{\tilde{\mu}[\mathcal{A}]}]$ and $x\in V$. Suppose that $t\in V\setminus\{x\}$. Assuming that $I(\mathcal{A})=\emptyset$, we can fix $U\in \mathcal{A}\setminus\{\emptyset\}$ such that $t\notin U$. Then $W=U\cup\{x\}\in\tilde{\mu}[\mathcal{A}]$, $x\in W$ but $t\notin W$, which contradicts the fact that $x\lsim_{\tilde{\mu}[\mathcal{A}]} t$. This contradiction shows that $V=\{x\}$. Hence $\{x\}\in \tau[\lsim_{\tilde{\mu}[\mathcal{A}]}]$ and, in consequence, the topology $\tau[\lsim_{\tilde{\mu}[\mathcal{A}]}]$ is discrete.
\end{proof}

For $x,y\in X$ such that $x\lsim_{\mathcal{A}} y$ and $y\lsim_{\mathcal{A}}x$, we write $x\approx_{\mathcal{A}}y$.

\begin{theorem}
\label{s3t10} 
Suppose that $I(\mathcal{A})\neq\emptyset$. Then $I(\mathcal{A})$ is the set of all weakly $\lsim_{\mathcal{A}}$-maximal elements and the following conditions are all equivalent:
\begin{enumerate}
\item[(i)] $\tau[\lsim_{\tilde{\mu}[\mathcal{A}]}]=\tilde{\mu}[\mathcal{A}]$;
\item[(ii)]  $I(\mathcal{A})\in\mu[\mathcal{A}]$;
\item[(iii)] $\tilde{\mu}[\mathcal{A}]=\super(I(\mathcal{A}))$;
\item[(iv)] $\tilde{\mu}[\mathcal{A}]$ is an Alexandroff topology.
\end{enumerate}
Furthermore, the set $I(\mathcal{A})$ is dense in $\tau[\lsim_{\tilde{\mu}[\mathcal{A}]}]$.
\end{theorem}
\begin{proof}
It is obvious that every element of $I(\mathcal{A})$ is weakly $\lsim_{\mathcal{A}}$-maximal. Furthermore, we can fix $x_0\in I(\mathcal{A})$. We notice that, for every $x\in X$, we have $x\lsim_{\mathcal{A}}x_0$. Hence, if $m\in X$ is a weakly $\lsim_{\mathcal{A}}$-maximal element, then $m\approx_{\mathcal{A}} x_0$, so $m\in I(\mathcal{A})$ because for the element $x_0$ we have both $x_0\lsim_{\mathcal{A}}m$ and $x_0\in I(\mathcal{A})$. This shows that $I(\mathcal{A})$ is the set of all weakly $\lsim_{\mathcal{A}}$-maximal elements. 

Now, let us assume that (i) holds. By Proposition \ref{s3p4}(4), $I(\tilde{\mu}[\mathcal{A}])\in\tau[\lsim_{\tilde{\mu}[\mathcal{A}]}]$.  Since $I(\mathcal{A})=I(\tilde{\mu}[\mathcal{A}])$ by Proposition \ref{s3p3}(2), we have $I(\mathcal{A})\in \tau[\lsim_{\tilde{\mu}[\mathcal{A}]}]$. It follows from (i) that $I(\mathcal{A})\in\tilde{\mu}[\mathcal{A}]$. By Proposition \ref{s3p3}(3), $I(\mathcal{A})\in\mu[\mathcal{A}]$. Hence (i) implies (ii). 

Since $I(\mathcal{A})\neq\emptyset$, it is obvious that (ii) implies (iii).

Suppose that (iii) holds. Let $\emptyset\neq U\in\tau[\lsim_{\tilde{\mu}[\mathcal{A}]}]$ and let $x\in U$. Put $V_x=I(\mathcal{A})\cup\{x\}$. It follows from (iii) that $V_x\in\tilde{\mu}[\mathcal{A}]$. One can easily check that $\uparrow[\lsim_{\tilde{\mu}[\mathcal{A}]},x]=V_x$. Hence $V_x\subseteq U$ because $x\in U\in \tau[\lsim_{\tilde{\mu}[\mathcal{A}]}]$. Since $V_x\in\tilde{\mu}[\mathcal{A}]$,  we deduce that $U\in\tilde{\mu}[\mathcal{A}]$.  Therefore, $\tau[\lsim_{\tilde{\mu}[\mathcal{A}]}]\subseteq\tilde{\mu}[\mathcal{A}]$.  Proposition \ref{s3p4}(5) completes the proof that (iii) implies (i). In this way, we have shown that conditions (i)--(iii) are all equivalent. Of course, (iii) implies (iv). Assuming that (iv) holds, we obtain that $I(\mathcal{A})\in\tilde{\mu}[\mathcal{A}]$. Thus, we infer from Proposition \ref{s3p3}(3) that $I(\mathcal{A})\in\mu[\mathcal{A}]$. Hence (iv) implies (ii) and, in consequence, conditions (i)--(iv) are all equivalent.

That $I(\mathcal{A})$ is dense in $\tau[\lsim_{\tilde{\mu}[\mathcal{A}]}]$ follows from the equivalence of (i) and (iii).
\end{proof}

\begin{corollary}
\label{s3c11} 
If $I(\mathcal{A})\neq\emptyset$ and $\tau[\lsim_{\mathcal{A}}]=\mu[\mathcal{A}]$, then $\tau[\lsim_{\tilde{\mu}[\mathcal{A}]}]=\tilde{\mu}[\mathcal{A}]$.
\end{corollary}
\begin{proof}
Assuming that $I(\mathcal{A})\neq\emptyset$ and $\tau[\lsim_{\mathcal{A}}]=\mu[\mathcal{A}]$, we deduce from Theorem \ref{s3t7} that $\mu[\mathcal{A}]$ is an Alexandroff topology. This implies that $I(\mathcal{A})\in\mu[\mathcal{A}]$. Theorem \ref{s3t10} completes the proof.
\end{proof}

The following example shows that the equality $\tau[\lsim_{\tilde{\mu}[\mathcal{A}]}]=\tilde{\mu}[\mathcal{A}]$ need not imply $\tau[\lsim_{\mathcal{A}}]=\mu[\mathcal{A}]$.

\begin{example}
\label{s3e12}
Suppose that $X$ is a set with at least two elements. Let $x_0$ be a fixed element of $X$ and let $\mathcal{A}=\{\{x_0\}\}$. Then $\mu[\mathcal{A}]=\{\emptyset, \{x_0\}\}$ is not a topology on $X$, so $\tau[\lsim_{\mathcal{A}}]\neq\mu[\mathcal{A}]$. One can also notice that $\lsim_{\mathcal{A}}=\{\langle x_0, x_0\rangle\}\cup((X\setminus\{x_0\})\times X)$, so $\tau[\lsim_{\mathcal{A}}]=\{\emptyset, X, \{x_0\}\}$. We have $I(\mathcal{A})=\{x_0\}$ and   $\tilde{\mu}[\mathcal{A}]=\super(I(\mathcal{A}))$. By Theorem \ref{s3t10}, $\tau[\lsim_{\tilde{\mu}[\mathcal{A}]}]=\tilde{\mu}[\mathcal{A}]$. That $\tau[\lsim_{\tilde{\mu}[\mathcal{A}]}]=\tilde{\mu}[\mathcal{A}]$ can be also deduced from the following equality:  $\lsim_{\tilde{\mu}[\mathcal{A}]}=\{\langle x, y\rangle\in X\times X: y\in\{x_0, x\}\}$.
\end{example}

\begin{theorem}
\label{s3t13}
Suppose that $I(\mathcal{A})\neq\emptyset$ and $\tau[\lsim_{\mathcal{A}}]=\tilde{\mu}[\mathcal{A}]$. Then $\lsim_{\mathcal{A}}=\lsim_{\tilde{\mu}[\mathcal{A}]}$.
\end{theorem}
\begin{proof}
It follows from Proposition \ref{s3p4}(4) that $I(\mathcal{A})\in\tilde{\mu}[\mathcal{A}]$. By Proposition \ref{s3p3}(3), $I(\mathcal{A})\in\mu[\mathcal{A}]$. Hence, by Theorem \ref{s3t10}, $\tau[\lsim_{\tilde{\mu}[\mathcal{A}]}]=\tilde{\mu}[\mathcal{A}]$. This, together with the equality $\tau[\lsim_{\mathcal{A}}]=\tilde{\mu}[\mathcal{A}]$, implies that $\tau[\lsim_{\tilde{\mu}[\mathcal{A}]}]=\tau[\lsim_{\mathcal{A}}]$. By Lemma \ref{s3l5}, $\lsim_{\mathcal{A}}=\lsim_{\tilde{\mu}[\mathcal{A}]}$.
\end{proof}

\begin{remark}
\label{s3r14}
Considering Theorem \ref{s3t9}, it is natural to ask what happens when $X$ is a non-empty set and $\mathcal{A}=\{\emptyset\}$. In this case, $I(\mathcal{A})$ is not defined, $\mu[\mathcal{A}]=\{\emptyset\}=\tilde{\mu}[\mathcal{A}]$, $\lsim_{\tilde{\mu}[\mathcal{A}]}=X\times X$, so $\tau[\lsim_{\tilde{\mu}[\mathcal{A}]}]$ is the indiscrete topology on $X$.
\end{remark}

By replacing the quasiorders involved in our results with their duals, one can easily obtain dual versions of the results.

\section{A characterization of iso-dense spaces}
\label{s4}

The main goal of this section is to apply the results of Section \ref{s3} to our investigation of iso-dense spaces via suitable quasiorders.
We use the notation introduced in Definitions \ref{s1d1}, \ref{s1d3}, \ref{s2d1} and \ref{s3d2}.

Let us recall that, given a generalized topological space $\mathbf{X}=\langle X, \mu\rangle$, $\Iso(\mathbf{X})$ is the set of all isolated points of $\mathbf{X}$, and $\mathcal{D}(\mathbf{X})$ is the collection of all dense sets in $\mathbf{X}$. Then 
 
\[\mu[\mathcal{D}(\mathbf{X})]=\tilde{\mu}[\mathcal{D}(\mathbf{X})]=\mathcal{D}(\mathbf{X})\cup\{\emptyset\}.\]

For the collection $\mathcal{DO}(\mathbf{X})$ of all dense open sets of $\mathbf{X}$, the generalized topology $\mu[\mathcal{DO}(\mathbf{X})]=\mathcal{DO}(\mathbf{X})\cup\{\emptyset\}$ need not be a topology on $\mathbf{X}$ (see Example \ref{s2e5}). We notice that 
\[\tilde{\mu}[\mathcal{DO}(\mathbf{X})]=\{\emptyset\}\cup\{U\in\mathcal{P}(X): \inter_{\mathbf{X}}(U)\in\mathcal{D}(\mathbf{X})\}.\]

In view of the assumptions of Section \ref{s3} and Definition \ref{s3d2}(i), the sets $I(\mathcal{D}(\mathbf{X}))$ and $I(\mathcal{DO}(\mathbf{X}))$ are defined only when $\mathcal{D}(\mathbf{X})\neq\{\emptyset\}\neq\mathcal{DO}(\mathbf{X})$.

As in Definition \ref{s2d1}(3), we denote by $\mathcal{ND}(\mathbf{X})$ the collection of all nowhere dense sets in $\mathbf{X}$. 

In Remarks \ref{s4r1} and \ref{s4r2} below, we show what happens with indiscrete generalized topologies.

\begin{remark}
\label{s4r1}
For a non-empty set $X$, let $\mu=\{\emptyset\}$ and $\mathbf{X}=\langle X, \mu\rangle$. Then the generalized topological space $\mathbf{X}$ is resolvable. Furthermore, $\mathcal{D}(\mathbf{X})=\mathcal{P}(X)=\mathcal{ND}(\mathbf{X})$, $\{\emptyset\}=\mathcal{DO}(\mathbf{X})=\mu[\mathcal{DO}(\mathbf{X})]=\tilde{\mu}(\mathcal{DO}(\mathbf{X}))$ and $\Iso(\mathbf{X})=\emptyset$. We notice that 
\[\lsim_{\mathcal{D}(\mathbf{X})}=\{\langle x, y\rangle\in X\times X: x=y\}\text{ and }\lsim_{\tilde{\mu}[\mathcal{DO}(\mathbf{X})]}=X\times X.\]
Therefore, $\tau[\lsim_{\mathcal{D}(\mathbf{X})}]=\mathcal{P}(X)$ and $\tau[\lsim_{\tilde{\mu}[\mathcal{DO}(\mathbf{X})]}]=\{\emptyset, X\}$. Moreover, the set $I(\mathcal{D}(\mathbf{X}))$ is defined, but $I(\mathcal{DO}(\mathbf{X}))$ is not defined. If $X$ is a singleton, then $I(\mathcal{D}(\mathbf{X}))=X$. If $X$ consists of at least two points, then $I(\mathcal{D}(\mathbf{X}))=\emptyset$.
\end{remark}

\begin{remark}
\label{s4r2}
For a non-empty set $X$, let $\mu=\{\emptyset, X\}$ and $\mathbf{X}=\langle X, \mu\rangle$. Then $\mathcal{D}(\mathbf{X})=\mathcal{P}(X)\setminus\{\emptyset\}$, $\mathcal{ND}(\mathbf{X})=\{\emptyset\}$, $\mathcal{DO}(\mathbf{X})=\{ X\}$, and $\tilde{\mu}[\mathcal{DO}(\mathbf{X})]=\mu[\mathcal{DO}(\mathbf{X})]=\mu$. If $X$ is a singleton, then $\mathbf{X}$ is irresolvable and $\Iso(\mathbf{X})=X=I(\mathcal{D}(\mathbf{X}))$. If $X$ consists of at least two points, then $\mathbf{X}$ is resolvable and $\Iso(\mathbf{X})=\emptyset=I(\mathcal{D}(\mathbf{X}))$. Moreover, $I(\mathcal{DO}(\mathbf{X}))=X$. We notice that 
\[\lsim_{\mathcal{D}(\mathbf{X})}=\{\langle x, y\rangle\in X\times X: x=y\}\text{ and }\lsim_{\tilde{\mu}[\mathcal{DO}(\mathbf{X})]}=X\times X.\]
Therefore, $\tau[\lsim_{\mathcal{D}(\mathbf{X})}]=\mathcal{P}(X)$ and $\tau[\lsim_{\tilde{\mu}[\mathcal{DO}(\mathbf{X})]}]=\mu$. 
\end{remark}

\begin{remark}
\label{s4r3}
Let $\mathbf{X}=\langle X, \mu\rangle$ be a non-indiscrete generalized topological space. Since $\mu\nsubseteq\{\emptyset, X\}$,  we have $\emptyset\neq X\in\mathcal{D}(\mathbf{X})$ and $\emptyset\neq\bigcup\mu\in\mathcal{DO}(\mathbf{X})$, so $I(\mathcal{D}(\mathbf{X}))$ and $I(\mathcal{DO}(\mathbf{X}))$ are both defined.
\end{remark}

In the sequel, we consider non-indiscrete generalized topologies.

\begin{proposition}
\label{s4p4}
For every non-indiscrete generalized topological space $\mathbf{X}$, the following conditions are satisfied:
\begin{enumerate}
\item[(i)] $I(\mathcal{D}(\mathbf{X}))=\Iso(\mathbf{X})$; in consequence, the set $I(\mathcal{D}(\mathbf{X}))$ is open in $\mathbf{X}$;
\item[(ii)] $I(\mathcal{DO}(\mathbf{X}))=\{x\in X: \{x\}\notin\mathcal{ND}(\mathbf{X})\}$;
\item[(iii)] $\Iso(\mathbf{X})\subseteq I(\mathcal{DO}(\mathbf{X}))$;
\item[(iv)] if $\mathbf{X}$ is either a $T_1$-space or an Alexandroff $T_0$-space, then $I(\mathcal{DO}(\mathbf{X}))=\Iso(\mathbf{X})$;
\item[(v)] $\mathbf{X}$ is iso-dense if and only if $\mu[\mathcal{D}(\mathbf{X})]=\super(I(\mathcal{D}(\mathbf{X}))$.
\end{enumerate}
\end{proposition}
\begin{proof}
Let $\mathbf{X}=\langle X, \mu\rangle$ be a non-indiscrete generalized topological space. Since $\Iso(\mathbf{X})$ is a subset of every dense set in $\mathbf{X}$, (iii) holds and $\Iso(\mathbf{X})\subseteq I(\mathcal{D}(\mathbf{X}))$. Furthermore, if $x\in X\setminus\Iso(\mathbf{X})$, then the set $X\setminus\{x\}$ is dense in $\mathbf{X}$ and, therefore, $I(\mathcal{D}(\mathbf{X}))\subseteq X\setminus\{x\}$. This implies that $I(\mathcal{D}(\mathbf{X}))\subseteq\Iso(\mathbf{X})$. Hence $I(\mathcal{D}(\mathbf{X}))=\Iso(\mathbf{X})$. Thus, since $\Iso(\mathbf{X})$ is open in $\mathbf{X}$, (i) holds. 

(ii) Let $x\in X$ be such that $\{x\}\in\mathcal{ND}(\mathbf{X})$. Then the set $X\setminus\cl_{\mathbf{X}}(\{x\})$ is non-empty, dense and open in $\mathbf{X}$, so $I(\mathcal{DO}(\mathbf{X}))\subseteq X\setminus\cl_{\mathbf{X}}(\{x\})$. This implies that $I(\mathcal{DO}(\mathbf{X}))\subseteq\{x\in X: \{x\}\notin\mathcal{ND}(\mathbf{X})\}$. To show that the reverse inclusion also holds, consider any set $D\in\mathcal{DO}(\mathbf{X})$ and any $x\in X$ with $\{x\}\notin\mathcal{ND}(\mathbf{X})$. Then the set $V=\inter_{\mathbf{X}}(\cl_{\mathbf{X}}(\{x\}))$ is a non-empty open set in $\mathbf{X}$. If $x\notin D$, then $\cl_{\mathbf{X}}(\{x\})\subseteq X\setminus D$ and $V\subseteq X\setminus D$. Since $V\cap D\neq\emptyset$, we deduce that $x\in D$. Hence $\{x\in X: \{x\}\notin\mathcal{ND}(\mathbf{X})\}\subseteq I(\mathcal{DO}(\mathbf{X}))$. This shows that (ii) holds.

(iv) Suppose that $\mathbf{X}$ is a $T_1$-space. Since every singleton of $\mathbf{X}$ is closed in $\mathbf{X}$, it is easily seen that $\{x\in X: \{x\}\notin\mathcal{ND}(\mathbf{X})\}=\Iso(\mathbf{X})$.

Now, suppose that $\mathbf{X}$ is an Alexandroff $T_0$-space. Consider any $x_0\in X$ such that $\{x_0\}\notin \mathcal{ND}(\mathbf{X})$. Let $W_0=\bigcap\{W\in\mu: x_0\in W\}$. Since $\mathbf{X}$ is an Alexandroff space, $W_0\in\mu$. Since $\{x_0\}\notin\mathcal{ND}(\mathbf{X})$, we have $x_0\in\inter_{\mathbf{X}}(\cl_{\mathbf{X}}(\{x_0\}))$. Therefore $W_0\subseteq\cl_{\mathbf{X}}(\{x_0\})$. Suppose that there exists $y\in W_0\setminus\{x_0\}$. Since $\mathbf{X}$ is a $T_0$-space and $y\in\cl_{\mathbf{X}}(\{x_0\})$, we have $x_0\notin\cl_{\mathbf{X}}(\{y\})$. This implies that $W_0\subseteq X\setminus \cl_{\mathbf{X}}(\{y\})$. But this is impossible for $y\in W_0$. The contradiction obtained shows that $W_0=\{x_0\}$. Hence $\{x\in X: \{x\}\notin\mathcal{ND}(\mathbf{X})\}\subseteq\Iso(\mathbf{X})$. This, together with (ii) and (iii), shows that (iv) holds.

That (v) holds follows from (i) and Proposition \ref{s2p7}.
\end{proof}

Among other things, we indicate which of the spaces from the examples given below are resolvable. The following example shows that even if an Alexandroff space is not indiscrete, it may happen that $I(\mathcal{D}(\mathbf{X}))\neq I(\mathcal{DO}(\mathbf{X}))$.

\begin{example}
\label{s4e5}
Let $X=\{1,2,3\}$, $\mu=\{\emptyset, X, \{1, 2\}\}$  and $\mathbf{X}=\langle X, \mu\rangle$. Then $\mathbf{X}$ is a non-indiscrete resolvable  space, $I(\mathcal{D}(\mathbf{X}))=\emptyset$ and $I(\mathcal{DO}(\mathbf{X}))=\{1,2\}$.
\end{example}

We show in the following example that,  for a non-indiscrete topological $T_0$-space $\mathbf{X}$, it may happen that $I(\mathcal{DO}(\mathbf{X}))\neq\Iso(\mathbf{X})$ and the set $I(\mathcal{DO}(\mathbf{X}))$ need not be open in $\mathbf{X}$.

\begin{example}
\label{s4e6}
Let $Y=\{0,1\}$, $\tau=\{\emptyset, Y, \{0\}\}$ and $\mathbf{Y}=\langle Y, \tau\rangle$. Let $\mathbf{X}=\mathbf{Y}^{\mathbb{N}}$. The topological space $\mathbf{X}$ is a dense-in-itself $T_0$-space which is not indiscrete. We have $\Iso(\mathbf{X})=\emptyset$. Let $x_0\in\{0,1\}^{\mathbb{N}}$ be defined by: for every $n\in\mathbb{N}$, $x_0(n)=0$. It is easily seen that $\{x_0\}=\{x\in X: \{x\}\notin\mathcal{ND}(\mathbf{X})\}$. Hence, by Proposition \ref{s4p4}(ii), $I(\mathcal{DO}(\mathbf{X}))=\{x_0\}$. The set $\{x_0\}$ is both dense and not open in $\mathbf{X}$.

Of course, the space $\mathbf{Y}$ is irresolvable. To see that $\mathbf{X}$ is resolvable, for every $k\in\mathbb{N}$, we define elements $y^{k}$ of $\{0, 1\}^{\mathbb{N}}$ as follows. For every $n\in\mathbb{N}$, we put:
\[ y^{k}(n)=\begin{cases} 0&\text{ if } n\leq k,\\
1&\text{ if } k<n \text{ and } n \text{ is even},\\
0&\text{ if } k<n \text{ and } n\text{ is odd}.\end{cases}\]

Let $D=\{y^k: k\in\mathbb{N}\}$. The sets $D$ and $\{0, 1\}^{\mathbb{N}}\setminus D$ are both dense in $\mathbf{X}$. This shows that $\mathbf{X}$ is resolvable.
\end{example}

\begin{theorem}
\label{s4t7}
For every non-indiscrete generalized topological space $\mathbf{X}$, the following conditions are satisfied:
\begin{enumerate}
\item[(i)] if $\mathbf{X}$ is dense-in-itself, then $\tau[\lsim_{\mathcal{D}(\mathbf{X})}]$ is the discrete topology on $X$;
\item[(ii)]  if every singleton of $\mathbf{X}$ is nowhere dense in $\mathbf{X}$, then $\tau[\lsim_{\tilde{\mu}[\mathcal{DO}(\mathbf{X})]}]$ is the discrete topology on $X$.
\end{enumerate}
\end{theorem}
\begin{proof} Let us fix a non-indiscrete generalized topological space $\mathbf{X}$.

 (i) Suppose that $\mathbf{X}$ is dense-in-itself. Then $\Iso(\mathbf{X})=\emptyset$, so $\tau[\lsim_{\mathcal{D}(\mathbf{X})}]$ is discrete by Theorem \ref{s3t9} and Proposition \ref{s4p4}(i).

(ii) Now, assume that all singletons of  $X$  are nowhere dense in $\mathbf{X}$. Then it follows from Proposition \ref{s4p4}(ii) that $I(\mathcal{DO}(\mathbf{X}))=\emptyset$. Hence, by Theorem \ref{s3t9}, (ii) holds.
\end{proof}

The following theorem characterizes non-indiscrete iso-dense generalized topological spaces and is the first main result of this section.

\begin{theorem}
\label{s4t8}
Let $\mathbf{X}=\langle X,\mu\rangle$ be a non-indiscrete generalized topological space. Then the following conditions are equivalent:
\begin{enumerate}
\item[(i)] $\tau[\lsim_{\mathcal{D}(\mathbf{X})}]=\mu[\mathcal{D}(\mathbf{X})]$;
\item[(ii)] $\mu[\mathcal{D}(\mathbf{X})]$ is an Alexandroff topology on $X$;
\item[(iii)] $\mathbf{X}$ is iso-dense.
\end{enumerate}
\end{theorem}
\begin{proof}
That conditions (i) and (ii) are equivalent follows from Theorem \ref{s3t7}. 

To show that (iii) implies (i), let us assume that the space $\mathbf{X}$ is iso-dense. Then the set $\Iso(\mathbf{X})$ is non-empty and dense in $\mathbf{X}$. By Proposition \ref{s4p4}(i), $\Iso(\mathbf{X})=I(\mathcal{D}(\mathbf{X}))$. Hence, $\emptyset\neq I(\mathcal{D}(\mathbf{X}))\in\mathcal{D}(\mathbf{X})$. Furthermore, we know that $\tilde{\mu}[\mathcal{D}(\mathbf{X})]=\mu[\mathcal{D}(\mathbf{X})]$. Thus, we can easily infer from Theorem \ref{s3t10} that (i) holds.

Now, let us prove that (i) implies (iii). First, suppose that $\mathbf{X}$ is dense-in-itself. Then, by Theorem \ref{s4t7}, $\tau[\lsim_{\mathcal{D}(\mathbf{X})}]$ is the discrete topology on $X$. Since $\mathbf{X}$ is not indiscrete, there exists $U\in\mu$ such that $\emptyset\neq U\neq X$. Let $C=X\setminus U$. Then $C\notin\mu[\mathcal{D}(\mathbf{X})]$ but $C\in\tau[\lsim_{\mathcal{D}(\mathbf{X})}]$. Hence, if (i) holds, the space $\mathbf{X}$ is not dense-in-itself. 

Let us assume that (i) is true. Then we have already shown that $\Iso(\mathbf{X})\neq\emptyset$. It follows from Proposition \ref{s4p4}(i) that $I(\mathcal{D}(\mathbf{X}))\neq\emptyset$. Since $\tilde{\mu}[\mathcal{D}(\mathbf{X})]=\mu[\mathcal{D}(\mathbf{X})]$, we deduce from Theorem \ref{s3t10} that $I(\mathcal{D}(\mathbf{X}))\in\mathcal{D}(\mathbf{X})$. Proposition \ref{s4p4}(i) completes the proof.
\end{proof}

The following theorem is the second main result of this section. It characterizes non-trivial generalized topological spaces $\mathbf{X}$ such that the set of all not nowhere dense singletons of $\mathbf{X}$ is both dense and open in $\mathbf{X}$. 

\begin{theorem}
\label{s4t9}
Let $\mathbf{X}=\langle X,\mu\rangle$ be a non-indiscrete generalized topological space. Then the following conditions are equivalent:
\begin{enumerate}
\item[(i)] $\tau[\lsim_{\tilde{\mu}[\mathcal{DO}(\mathbf{X})]}]=\tilde{\mu}[\mathcal{DO}(\mathbf{X})]$;
\item[(ii)] $\{x\in X: \{x\}\notin\mathcal{ND}(\mathbf{X})\}\in\mathcal{DO}(\mathbf{X})$;
\item[(iii)] $\tilde{\mu}[\mathcal{DO}(\mathbf{X})]=\super(\{x\in X: \{x\}\notin\mathcal{ND}(\mathbf{X})\})$; 
\item[(iv)] $\tilde{\mu}[\mathcal{DO}(\mathbf{X})]$ is an Alexandroff topology on $X$.
\end{enumerate}
\end{theorem}
\begin{proof}
We know from Proposition \ref{s4p4}(ii) that $I(\mathcal{DO}(\mathbf{X}))=\{x\in X: \{x\}\notin\mathcal{ND}(\mathbf{X})\}$. That conditions (i) and (iv) are equivalent follows from Theorem \ref{s3t7}. As in the proof of Theorem \ref{s4t8}, we fix a set $U\in\mu$ such that $\emptyset\neq U\neq X$, and put  $C=X\setminus U$. We notice that $C\notin\mathcal{D}(\mathbf{X})$, so $C\notin\tilde{\mu}[\mathcal{DO}(\mathbf{X})]$. 

Suppose that $I(\mathcal{DO}(\mathbf{X}))=\emptyset$. Then $\super(I(\mathcal{DO}(\mathbf{X})))$ is discrete, so (iii) is false. By Theorem \ref{s4t7}, the topology $\tau[\lsim_{\tilde{\mu}[\mathcal{DO}(\mathbf{X})]}]$ is also discrete. Hence (i) is also false. In consequence, each of the conditions (i) and (iii) implies that $I(\mathcal{DO}(\mathbf{X}))\neq\emptyset$. Of course, (ii) also implies that $I(\mathcal{DO}(\mathbf{X}))\neq\emptyset$. Therefore, that conditions (i), (ii) and (iii) are equivalent follows from Theorem \ref{s3t10}.
\end{proof}

\begin{remark}
\label{s4r10}
Let $X$ be a non-empty set, $\mu=\{\emptyset\}$ and $\mathbf{X}=\langle X, \mu\rangle$ (see Remark \ref{s4r1}). Then $\Iso(\mathbf{X})=\emptyset$, but $\mathbf{X}$ is iso-dense and $\emptyset=\{x\in X: \{x\}\notin\mathcal{ND}(\mathbf{X})\}\in\mathcal{DO}(\mathbf{X})$. However, since $X\notin\mu$, none of the conditions (i), (iii) and (iv) of Theorem \ref{s4t9} is satisfied.
\end{remark}

\begin{remark}
\label{s4r11}
Let $X$ be a non-empty set, $\mu=\{\emptyset, X\}$ and $\mathbf{X}=\langle X, \mu\rangle$ (see Remark \ref{s4r2}). Then all the conditions (i)--(iv) of Theorem \ref{s4t9} are satisfied.
\end{remark}

\begin{corollary}
\label{s4c12}
Let $\mathbf{X}=\langle X, \mu\rangle$ be a $T_1$-space such that $\mu\neq\{\emptyset\}$. Then that $\mathbf{X}$ is iso-dense is equivalent to each of the conditions (i)-(iv) of Theorem \ref{s4t9}. 
\end{corollary}
\begin{proof}
If $X$ consists of at least two points, then $\mathbf{X}$ is not indiscrete, so the result follows from  Theorem \ref{s4t9} taken together with items (ii) and (iv) of Proposition \ref{s4p4}.

Suppose that $X$ is either empty or a singleton. Then, since $\mu\neq\{\emptyset\}$, we have $\mu=\{\emptyset, X\}$,  $\mathbf{X}$ is iso-dense and, in view of Remark \ref{s4r11}, each of the conditions (i)--(iv) of Theorem \ref{s4t9} is satisfied. 
\end{proof}

\section{Final comments on Questions \ref{s1q7} and \ref{s1q8}}
\label{s5}

Theorem \ref{s4t8} and Corollary \ref{s4c12} answer Questions \ref{s1q8} and \ref{s2q4}, and can be regarded as partial answers to Question \ref{s1q7}. To search for a more satisfactory answer to Question \ref{s1q7}, one needs to consider the formulae $\mathbf{F}_d$ and $\mathbf{F}_d^T$ defined as follows.

\begin{definition}
\label{s5d1} 
Let $\mathbf{X}=\langle X, \mu\rangle$ be a generalized topological space. 
\begin{enumerate}
\item[(i)] $\mathbf{F}_d(\mathbf{X})$  and $\mathbf{F}_d(\mu)$ denote the statement: for every pair $A, B$ of dense sets of $\mathbf{X}$, the set $A\cap B$ is dense in $\mathbf{X}$.
\item[(ii)] $\mathbf{F}_d^{T}(\mathbf{X})$ and $\mathbf{F}_d^T(\mu)$ denote the statement:  the generalized topology $\mu[\mathcal{D}(\mathbf{X})]$ is a topology on $\mathbf{X}$. 
\end{enumerate}
\end{definition}

\begin{remark}
\label{s5r2}
\begin{enumerate}
\item[(a)] Note that $\mathbf{F}_d^{T}(\mathbf{X})$ is equivalent to the statement: for every pair $A, B$ of dense sets of $\mathbf{X}$, the set $A\cap B$ is empty or dense in $\mathbf{X}$. (See Remark \ref{s2r3}.)
\item[(b)] In general, even for a topological space $\mathbf{X}$, $\mathbf{F}_d^{T}(\mathbf{X})$ need not imply $\mathbf{F}_d(\mathbf{X})$. To see this, we assume that $X$ is a set consisting of at least two points, $\tau=\{\emptyset, X\}$ and $\mathbf{X}=\langle X, \tau\rangle$ (see Remark \ref{s4r2}). For the indiscrete resolvable space $\mathbf{X}=\langle X, \tau\rangle$,  $\mathbf{F}_d^T(\mathbf{X})$ holds but $\mathbf{F}_d(\mathbf{X})$ does not.
\end{enumerate}
\end{remark}

The following proposition is straightforward.

\begin{proposition}
\label{s5p3}
For every generalized topological space $\mathbf{X}$, the following conditions are satisfied:
\begin{enumerate}
\item[(i)] if $\mathbf{X}$ is iso-dense, then $\mathbf{F}_d(\mathbf{X})$ holds;
\item[(ii)] $\mathbf{F}_d(\mathbf{X})$ implies $\mathbf{F}_d^{T}(\mathbf{X})$.
\end{enumerate}
\end{proposition}

\begin{proposition}
\label{s5p4}
Let $\mathbf{X}=\langle X, \mu\rangle$ be generalized topological space such that $\mu\neq\{\emptyset, X\}$. Then $\mathbf{F}_d^T(\mathbf{X})$ and $\mathbf{F}_d(\mathbf{X})$ are equivalent. Furthermore, if the space $\mathbf{X}$ is resolvable and non-indiscrete, then $\mathbf{F}_d^T(\mathbf{X})$ is false.
\end{proposition}
\begin{proof}
If $\mu=\{\emptyset\}$, then both $\mathbf{F}_d^T(\mathbf{X})$ and $\mathbf{F}_d(\mathbf{X})$ are true, so equivalent. 

Assume that $\mu\neq\{\emptyset\}$. Since $\mu\neq\{\emptyset, X\}$, we can choose a point $x_0\in X$ such that the set $\{x_0\}$ is not dense in $\mathbf{X}$. Suppose that $A, B$ is a pair of disjoint dense sets in $\mathbf{X}$. Let $C=B\cup\{x_0\}$ and $D=B\cup\{x_0\}$. The sets $C$ and $D$ are both dense in $\mathbf{X}$ but $C\cap D=\{x_0\}$ is not dense in $\mathbf{X}$. This shows that if $\mathbf{F}_d^T(\mathbf{X})$ holds, so does $\mathbf{F}_d(\mathbf{X})$; moreover, if $\mathbf{X}$ is resolvable, then $\mathbf{F}_d^T(\mathbf{X})$ is false. Proposition \ref{s5p3}(ii) completes the proof.
\end{proof}

\begin{corollary}
\label{s5c5}
For every generalized topological $T_0$-space $\mathbf{X}$, the statements $\mathbf{F}_d(\mathbf{X})$ and $\mathbf{F}_d(\mathbf{Y})$ are equivalent.
\end{corollary}
\begin{proof}
Let $\mathbf{X}=\langle X, \mu\rangle$ be a generalized topological $T_0$-space. If $\mu=\{\emptyset, X\}$, then $X$ is either empty or a singleton, so $\mathbf{F}_d^T(\mathbf{X})$ and $\mathbf{F}_d(\mathbf{X})$ are both true. If $\mu\neq\{\emptyset, X\}$, it suffices to apply Proposition \ref{s5p4}.
\end{proof}

\begin{corollary}
\label{s5c6}
If $\mathbf{X}=\langle X, \mu\rangle$ is a generalized topological space such that $\mathbf{F}_d^T(\mathbf{X})$ is true and $\mathbf{F}_d(\mathbf{X})$ is false, then the set $X$ consists of at least two points and $\mu=\{\emptyset, X\}$.
\end{corollary}

It follows from Proposition \ref{s5p4} that, for the space $\mathbf{X}$ from Example \ref{s4e6}, $\mathbf{F}_d(\mathbf{X})$ and $\mathbf{F}_d^T(\mathbf{X})$ are both false.

\begin{proposition}
\label{s5p7}
Let $\mathbf{X}$ be a locally compact, dense subspace of a Hausdorff topological space $\mathbf{Y}$. Then the following conditions are satisfied:
\begin{enumerate}
\item[(i)] $\mathbf{F}_d(\mathbf{X})$, $\mathbf{F}_d(\mathbf{Y})$, $\mathbf{F}_d^T(\mathbf{X})$ and $\mathbf{F}_d^T(\mathbf{Y})$ are all equivalent;
\item[(ii)] $\mathbf{X}$ is resolvable if and only if $\mathbf{Y}$ is resolvable.
\end{enumerate}
\end{proposition}
\begin{proof}
Let $A,B\in\mathcal{D}(\mathbf{Y})$ and $C,D\in\mathcal{D}(\mathbf{X)}$. Then both $A\cap X$ and $B\cap X$ are dense in $\mathbf{X}$ because $X$ is open in $\mathbf{Y}$ (see \cite[Theorem 3.3.9]{En}). Furthermore, since $X$ is dense in $\mathbf{Y}$, the sets $C$ and $D$ are both dense in $\mathbf{Y}$. If $C\cap D=\emptyset$, then $\mathbf{Y}$ is resolvable. If $A\cap B=\emptyset$, then $\mathbf{X}$ is resolvable. Hence (ii) holds. 

Let us prove that (i) is satisfied. To this aim, by Corollary \ref{s5c5}, it suffices to show that $\mathbf{F}_d(\mathbf{X})$ and $\mathbf{F}_d(\mathbf{Y})$ are equivalent. We observe that if $\mathbf{F}_d(\mathbf{X})$ is true, then $A\cap B\cap X$ is dense in $\mathbf{X}$, so $A\cap B$ is dense in $\mathbf{Y}$. Therefore, $\mathbf{F}_d(\mathbf{X})$ implies $\mathbf{F}_d(\mathbf{Y})$.

Now, suppose that $\mathbf{F}_d(\mathbf{Y})$ is true. Then the set $C\cap D$ is dense in $\mathbf{Y}$. Therefore, since $C\cap D\subseteq X$, the set $C\cap D$ is dense in $\mathbf{X}$. Hence $\mathbf{F}_d(\mathbf{Y})$  implies $\mathbf{F}_d(\mathbf{X})$. This completes the proof.
\end{proof}

\begin{corollary}
\label{s5c8} For every Hausdorff compactification $\alpha\mathbb{R}$ of the real line $\mathbb{R}$ equipped with the natural topology, it holds that $\alpha\mathbb{R}$ is resolvable and $\mathbf{F}_d(\alpha\mathbb{R})$ is false. 
\end{corollary}

In the light of Proposition \ref{s5p4} and Corollary \ref{s5c6}, a satisfactory answer to the following question will be also a satisfactory answer to Question \ref{s1q7}.

\begin{question}
\label{s5q9}
Under which conditions on an irresolvable generalized topological space $\mathbf{X}$ is $\mathbf{F}_d(\mathbf{X})$ true?
\end{question}

Of course, for every iso-dense space $\mathbf{X}$, $\mathbf{F}_d(\mathbf{X})$ is true; however, to be an iso-dense space is not a necessary condition for an irresolvable generalized topological space $\mathbf{X}$ to satisfy $\mathbf{F}_d(\mathbf{X})$.  To investigate which non-trivial irresolvable generalized topological spaces $\mathbf{X}$ satisfy $\mathbf{F}_d(\mathbf{X})$ in $\mathbf{ZF}$ is a good topic for extensive future research.  Let us finish this article with the following theorems concerning cofinite topologies on infinite sets. We recall that, for any set $X$, $\tau_{cof}(X)$ denotes the cofinite topology on $X$ (see Definition \ref{s2d14}).

\begin{theorem}
\label{s5t10}
For every infinite set $X$, the following conditions are equivalent:
\begin{enumerate}
\item[(i)] $\mathbf{F}_d(\tau_{cof}(X))$ is true;
\item[(ii)] the cofinite topology $\tau_{cof}(X)$ on $X$ is irresolvable;
\item[(iii)] $X$ is an amorphous set.
\end{enumerate}
\end{theorem}
\begin{proof} We fix an infinite set $X$. That (i) implies (ii) follows from Proposition \ref{s5p4}. If $X$ is not amorphous, there exists an infinite subset $D$ of $X$ such that $X\setminus D$ is also infinite. Then both the sets $D$ and $X\setminus D$ are $\tau_{cof}(X)$-dense.  Hence (ii) implies (iii).  

Assuming that (iii) holds, we consider any pair $A, B$ of $\tau_{cof}(X)$-dense sets. Then both $A$ and $B$ are infinite subsets of $X$. Since $X$ is amorphous, the sets $X\setminus A$ and $X\setminus B$ are both finite. This implies that the set $X\setminus (A\cap B)=(X\setminus A)\cup (X\setminus B)$ is finite. Since $X$ is infinite, we infer that $A\cap B$ is infinite. Therefore, $A\cap B$ is $\tau_{cof}(X)$-dense. In consequence,(iii) implies (i).
\end{proof}

The following theorem is an immediate consequence of Theorem \ref{s5t10} and Proposition \ref{s5p4}.

\begin{theorem}
\label{s5t11}
It holds in $\mathbf{ZF}$ that the following statements are all equivalent:
\begin{enumerate}
\item[(i)] there are no amorphous sets (\cite[Form 64]{HR});
\item[(ii)] for every infinite set $X$, $\tau_{cof}(X)$ is resolvable;
\item[(iii)] for every infinite set $X$, $\mathbf{F}_d(\tau_{cof}(X))$ is false;
\item[(iv)]  for every infinite set $X$, $\mathbf{F}_d^T(\tau_{cof}(X))$ is false.
\end{enumerate}
\end{theorem}
 
Since Form 64 of \cite{HR} is known to be independent of $\mathbf{ZF}$ (see Remark \ref{s2r16}), we deduce from Theorem \ref{s5t11} the following independence results: 
 
\begin{corollary}
\label{s5c12}
The statements (i)--(iv) of Theorem \ref{s5t11} are all independent of $\mathbf{ZF}$.
\end{corollary}

\section{Open problems}
\label{s6}

We do not have satisfactory solutions to the following open problems.
\begin{enumerate}
\item Find necessary and sufficient conditions in $\mathbf{ZF}$ for an irresolvable (generalized) topological space $\mathbf{X}$ to satisfy $\mathbf{F}_d(\mathbf{X})$. (See Question \ref{s1q7} and \ref{s5q9}.)
\item Find necessary and sufficient conditions in $\mathbf{ZF}$ for a generalized topological space $\mathbf{X}$ to be such that $\mathcal{DO}(\mathbf{X})\cup\{\emptyset\}$ is a topology on $X$. (See Example \ref{s2e5} and Proposition \ref{s2p6}.)

\item Verify which known $\mathbf{ZFC}$-theorems on resolvable spaces are independent of $\mathbf{ZF}$.

\end{enumerate}

\end{document}